\documentclass[11pt]{amsart}

\usepackage{amssymb}
\usepackage{mathpazo}

\usepackage{mathrsfs}

\usepackage[all]{xy}
\usepackage[headings]{fullpage}
\usepackage{paralist}

\usepackage[colorlinks=true,linkcolor=blue, citecolor=blue, %
                urlcolor=blue, pagebackref=true]{hyperref}
\makeatletter

\newcounter{maintheorem}[equation]
\def\themaintheorem{\thesection.\@arabic \c@maintheorem}
\@addtoreset{equation}{section}
\def\theequation{\thesection.\@arabic \c@equation}
\def\theenumi{\@alph\c@enumi}
\def\theenumii{\@roman\c@enumii}

\makeatother

\theoremstyle{plain}
\newtheorem{theorem}[equation]{Theorem}
\newtheorem{lemma}[equation]{Lemma}
\newtheorem{corollary}[equation]{Corollary}
\newtheorem{proposition}[equation]{Proposition}

\theoremstyle{definition}

\newtheorem{remark}[equation]{Remark}
\newenvironment{remarkbox}[1][]{%
    \begin{remark}[#1] \pushQED{\qed}}{\popQED \end{remark}}

\newtheorem{example}[equation]{Example}
\newenvironment{examplebox}[1][]{%
    \begin{example}[#1] \pushQED{\qed}}{\popQED \end{example}}

\newtheorem{definition}[equation]{Definition}

\newtheorem{notation}[equation]{Notation}
\newenvironment{notationbox}[1][]{%
    \begin{notation}[#1]\pushQED{\qed}}{\popQED \end{notation}}

\newtheorem{discussion}[equation]{Discussion}
\newenvironment{discussionbox}[1][]{%
    \begin{discussion}[#1]\pushQED{\qed}}{\popQED \end{discussion}}

\newtheorem{observation}[equation]{Observation}

\newtheorem{construction}[equation]{Construction}

\newcommand{\calF}{\mathcal F}

\newcommand{\calG}{\mathcal G}

\newcommand{\frakm}{{\mathfrak m}}

\newcommand{\frakn}{{\mathfrak n}}

 \let\strSh\scrO

\DeclareMathOperator{\scrR}{\mathscr R}

\newcommand{\calZ}{\mathcal Z}

\newcommand{\naturals}{\mathbb{N}}
\newcommand{\ints}{\mathbb{Z}}

\newcommand{\reals}{\mathbb{R}}
\newcommand{\complex}{\mathbb{C}}

\newcommand\hssymb{\mathfrak z}
\def\to{\longrightarrow}

\DeclareMathOperator{\length}{length}
\DeclareMathOperator{\coker}{coker}

\DeclareMathOperator{\projective}{\mathbb{P}}

\DeclareMathOperator{\Hom}{Hom}
\DeclareMathOperator{\Tor}{Tor}

\DeclareMathOperator{\Spec}{Spec}
\DeclareMathOperator{\Proj}{Proj}
\DeclareMathOperator{\depth}{depth}
\DeclareMathOperator{\type}{type}
\DeclareMathOperator{\homology}{H}
\newcommand{\define}[1]{\emph{#1}}
\newcommand{\minus}{\ensuremath{\smallsetminus}}
\DeclareMathOperator{\image}{Im}
\DeclareMathOperator{\socle}{soc}
\DeclareMathOperator{\Ext}{Ext}

\makeatletter
\def\RDerChar{\mathbf{R}}
\def\RDer{\@ifnextchar[{\R@Der}{\ensuremath{\RDerChar}}}
\def\R@Der[#1]{\ensuremath{\RDerChar^{#1}}}
\makeatother

\begin{document}

\title[Conjectures of Itoh and of Lipman]
{On conjectures of Itoh and of Lipman on the cohomology of normalized
blow-ups}

\author[Kummini]{Manoj Kummini}
\address{Chennai Mathematical Institute, Siruseri, Tamilnadu 603103. India}
\email{mkummini@cmi.ac.in}

\author[Masuti] {Shreedevi K. Masuti}
\address{Institute of Mathematical Sciences, Chennai, Tamilnadu 600113. India}
\email{shreedevikm@imsc.res.in}

\date{\today}

\thanks{MK was partially supported by a grant from Infosys Foundation.
SKM was supported by Department of Atomic Energy, Government of India.} 



\begin{abstract}
Let $(R, \frakm, \Bbbk)$ be a Noetherian three-dimensional Cohen-Macaulay
analytically unramified ring and $I$ an $\frakm$-primary $R$-ideal.
Write $X = \Proj\left(\oplus_{n \in \naturals} \overline{I^n}t^n\right)$.
We prove some consequences of the vanishing of $\homology^2(X, \strSh_X)$,
whose length equals the the constant term $\bar e_3(I)$ of the normal
Hilbert polynomial of $I$. Firstly, $X$ is Cohen-Macaulay. Secondly, if the
extended Rees ring
$A := \oplus_{n \in \ints} \overline{I^n}t^n$ is not Cohen-Macaulay, and
either $R$ is equicharacteristic or $\overline{I} = \frakm$, then 
$\bar e_2(I) -
\length_R\left(\frac{\overline{I^2}}{I\overline{I}}\right) \geq
3$; this estimate is proved using Boij-S\"oderberg theory of coherent
sheaves on $\projective^2_\Bbbk$. The two results above are related to a
conjecture of S.~Itoh (J.~Algebra, 1992). 
Thirdly, 
$\homology^2_E(X, I^m\strSh_X) = 0$ for all integers $m$,
where $E$ is the exceptional divisor in $X$.
Finally, if additionally $R$ is regular and $X$ is pseudo-rational, then the
adjoint ideals $\widetilde{I^n}, n \geq 1$ satisfy 
$\widetilde{I^n} = I\widetilde{I^{n-1}}$ for all $n \geq 3$. The last two
results are related to conjectures of J.~Lipman (Math. Res. Lett., 1994).
\end{abstract}

\maketitle 

\section{Results}
\label{sec:results}

Let $R$ be a commutative Noetherian local ring and 
$R = I_0 \supseteq I_1 \supseteq I_2 \supseteq \cdots$ a filtration of $R$
by ideals.
Several authors have studied homological properties of the \define{Rees ring} 
$\oplus_{n \in \naturals}I_nt^n$ 
of this filtration ($t$ is an indeterminate). 
Duong Qu{\^o}c Vi{\^e}t~\cite[Theorem~1.1]{VietCMRees1993},
and S.~Goto and K.~Nishida~\cite[Theorem~(1.1)]{GotoNishida1994}
independently determined a criterion for the
Cohen-Macaulayness of the Rees ring (under a mild hypothesis) in terms of
the local cohomology modules of the \define{associated graded ring}
$\oplus_{n \in \naturals}(I_n/I_{n+1})t^n$;
see also~\cite[Theorem~4.4]{LipCMgraded94}.

Take a Noetherian three-dimensional Cohen-Macaulay 
local ring $(R, \frakm, \Bbbk)$ 
whose completion has no nilpotents (i.e., $R$ is
\define{analytically unramified})
and an $\frakm$-primary $R$-ideal $I$, and consider the filtration 
$\overline{I^n}, n \in \naturals$; here $\overline{(.)}$ denotes
integral closure. In this situation, the criterion of 
Vi{\^e}t and Goto-Nishida is that the Rees ring 
$\scrR : =\oplus_{n \in \naturals} \overline{I^n}t^n$
is Cohen-Macaulay if and only if the
associated graded ring 
$G := \oplus_{n \in \naturals}(\overline{I^n}/\overline{I^{n+1}})t^n$
is Cohen-Macaulay and $\homology^3_{G_+}(G)_n = 0$ for every $n \geq 0$.
($G_+$ is the $G$-ideal $\oplus_{n \geq 1}G_n$.)

Write $X$ for $\Proj \scrR$.
There is a natural projective birational morphism $f : X \to \Spec R$.
The sheaf of ideals $I\strSh_X := 
\image\left(I  \otimes_R \strSh_X \to \strSh_X\right)$ on $X$ is invertible
and ample,
and defines an effective Cartier divisor $E \subseteq X$. Note that 
$E = \Proj G$ and that it is the exceptional locus of
$f$. For any coherent sheaf $\calF$ on $X$ and $m \in \ints$, we will write
$\calF(m) = \calF \otimes_{\strSh_X} (I\strSh_X)^m$. One can show that 
$\homology^3_{G_+}(G)_n = \homology^2(E, \strSh_E(n)) = 0$ for every $n
\geq 0$ if and only if $\homology^2(X, \strSh_X) = 0$. Hence, assuming that 
$\homology^2(X, \strSh_X) = 0$, one may ask for sufficient and necessary
conditions for $G$ to be Cohen-Macaulay.

There is another way of looking at this. Write
\[
\sum_{i=0}^3 (-1)^i {\bar e_i}(I)\binom{n+3-i}{3-i}
\]
for the \define{normal Hilbert polynomial} of $I$, i.e., the polynomial
function that gives $\length_R (R/\overline{I^{n+1}})$ for all
sufficiently large $n$. The coefficients ${\bar e_i}(I)$ are integers, and are
often called the \define{normal Hilbert coefficients} of $I$.
Let $A := \oplus_{n \in \ints} \overline{I^n}t^n$ be the \define{extended
Rees ring} of this filtration.  Note that $G = A/(t^{-1})$; hence $G$ is
Cohen-Macaulay if and only if $A$ is Cohen-Macaulay. 
Let $\frakn$ be the $A$-ideal $ItA + (t^{-1})A$.
S.~Itoh~\cite[p.~114]{ItohNormalHilbCoeffs92} showed that
$\bar{e}_3(I) = \length_R(\homology^2(X, \strSh_X))$
(Note that for every coherent $\calF$ on $X$, $\homology^i(X, \calF)$ is a
finite-length $R$-module for every $i > 0$.)
He conjectured that if $R$ is additionally Gorenstein and 
$\bar{e}_3(I) = 0$, then $A$ is Cohen-Macaulay~\cite[Conjecture, p.~116
and Corollary~16]{ItohNormalHilbCoeffs92}.

Our results, of which the following theorem is the principal one, concern
the vanishing of $\bar{e}_3(I) = \length_R(\homology^2(X, \strSh_X))$. 
For the benefit of the reader, and for putting certain assertions in 
context, we have stated some of them in various equivalent forms.

\begin{theorem}
\label{theorem:itoh}
Suppose that $\bar e_3(I) = 0$ (i.e., $\homology^2(X, \strSh_X) = 0$).
Then:
\begin{enumerate}

\item \label{enum:thmitohHthreelocalvanishes}
$\homology^3_\frakn(A) = 0$, and hence $X$ and $E$ are Cohen-Macaulay,
i.e., every local ring of these schemes is Cohen-Macaulay.

\item \label{enum:thmitohboundForEtwoBar}
$\bar{e}_2(I)=
\length_R(\homology^2(E, \strSh_E(-1))) =
\length_R(\homology^2(X, \strSh_X(-1)))$.
Further, there is an inclusion $\homology^2(X, \strSh_X(-1)) \subseteq
(0 :_{\homology^3_\frakm(R)} \overline{I})$.

\item \label{enum:thmitohexactseq}
Suppose that $I$ has a reduction generated by three elements $x,y,z$.
Then there exist exact sequences (which are restatements of each other):
\[
\xymatrix@R=1ex{
0 \ar[r] & \frac{\overline{I^2}}{(x,y,z)\overline{I}} \ar[r]  &  
\homology^2(X, \strSh_X(-1)) \ar[r]& 
\homology^1(X, \strSh_X(1))^{\oplus 3} \ar[r]^-\phi& 
\homology^1(X, \strSh_X(2)) \ar[r] &  0; \\
0 \ar[r] & \frac{\overline{I^2}}{(x,y,z)\overline{I}} \ar[r]  &  
\homology^3_{It}(A)_{-1} \ar[r]& 
\left(\homology^2_{It}(A)_{1}\right) ^{\oplus 3} \ar[r]^-\phi& 
\homology^2_{It}(A)_{2} \ar[r] &  0.
}
\]
where $\phi$ arises from the Koszul complex on $xt,yt,zt \in A$. Further, $A$ is
Cohen-Macaulay if and only if $\homology^1(X, \strSh_X(1))=0$.

\item \label{enum:thmitohEquichar}
Additionally suppose either that $R$ is equicharacteristic or that
$\overline{I} = \frakm$. If $A$ is not Cohen-Macaulay, then $\length_R(\ker
\phi) \geq 3$.
\end{enumerate}
\end{theorem}

That $\length_R(\ker \phi) > 1$ if it is non-zero 
(without the hypothesis of
Theorem~\ref{theorem:itoh}\eqref{enum:thmitohEquichar})
can also be obtained from~\cite[Theorem~2]{ItohNormalHilbCoeffs92}.
The extra hypothesis of
Theorem~\ref{theorem:itoh}\eqref{enum:thmitohexactseq} is not
really a restriction. 
We may replace $R$ by a suitable extension ring (so
that the residue field is infinite), and $I$ by the extended ideal, without
effecting the numerical data that is part of the hypothesis.

Itoh proved a special case of the conjecture: with
$\overline{I} = \frakm$~\cite[Theorem~3]{ItohNormalHilbCoeffs92}.
Recently A.~Corso, C.~Polini and
M.~E.~Rossi~\cite[Theorem~3.3]{CorsoPoliniRossiNormalHilbCoeffs2014}
extended Itoh's result (with $\overline{I} = \frakm$) to the case of
arbitrary Cohen-Macaulay $R$ and $I$ satisfying
$\type(R) \leq \length_R(\overline{I^2}/\frakm I) + 1$. 
(Here $\type(R)$ is the \define{type} of $R$, i.e., 
$\dim_\Bbbk \Ext_R^3(\Bbbk, R)$, which also equals 
$\dim_\Bbbk \socle(\homology^3_\frakm(R))$.)
The following corollary of Theorem~\ref{theorem:itoh}
strengthens this result :
\begin{corollary}
\label{corollary:cpr}
Suppose that $\bar e_3(\frakm) = 0$.
\begin{enumerate}
\item %
$\bar e_2(\frakm) \leq \type(R)$.
\item %
$A$ is Cohen-Macaulay if
$\bar e_2(\frakm) \leq \length_R(\overline{I^2}/\frakm I) + 2$
for any ideal $I$ such that $\overline{I} = \frakm$. 
\end{enumerate}
\end{corollary}

Theorem~\ref{theorem:itoh} does not settle Itoh's conjecture. 
We hope that proving increasingly better lower bounds for
$\length_R(\ker \phi)$
would give a method of checking whether $A$ is Cohen-Macaulay.
General lower bounds for the coefficients $\bar e_i(I)$ (especially for
$i=1,2$) have been established by Itoh~\cite{ItohIntClRegSeq88,
ItohNormalHilbCoeffs92}, C.~Huneke~\cite{HuneHilbFnSymbPwrs1987} and 
Corso-Polini-Rossi~\cite{CorsoPoliniRossiNormalHilbCoeffs2014}.

We now exhibit the relation of the hypothesis that 
$\length_R(\homology^2(X, \strSh_X))=0$ to two conjectures of J.~Lipman. 
Suppose that $S$ is a regular
local ring.
Denote the \define{adjoint}~\cite[Definition (1.1)]{LipmanAdjoints94} of an
$S$-ideal $J$ by $\widetilde{J}$.
In~\cite[Conjecture~(1.6)]{LipmanAdjoints94}, he asserted that 
$\widetilde{J^n} = J\widetilde{J^{n-1}}$ for all $n \geq \ell(J)$, where
$\ell(J)$ is the analytic spread of $J$. 
In~\cite[Conjecture~(2.2)]{LipmanAdjoints94} (from which the first
conjecture would follow), he asserted that if $Y$ is any regular $S$-scheme
that is
obtained by a sequence of blow-ups at non-singular centres and on which
$J\strSh_Y$ is invertible, then with $Y_0$ denoting the closed fibre 
of the morphism $Y \to \Spec S$, $\homology^i_{Y_0}(Y, J\strSh_Y^{-1}) = 0$
for all $i < \dim S$; see~\cite[Remark~(2.2.1)(b)]{LipmanAdjoints94}.
In dimension two, Lipman observes~\cite[(2.2.1)]{LipmanAdjoints94} that
these conjectures follow from his earlier work.
In dimensions three and higher,~\cite[Conjecture~(2.2)]{LipmanAdjoints94}
is known to hold for essentially finite-type rings over a field of
characteristic zero (proofs by 
S.~D.~Cutkosky in~\cite[Appendix A]{LipmanAdjoints94},
using the Kodaira-Ramanujam vanishing theorem, and by E.~Hyry and
O.~Villamayor~\cite [Theorem~2.7]{HyryVillamayorBrianconSkoda1998}, using
the Grauert-Riemenschneider vanishing theorem) and for finitely supported
ideals in regular local rings of arbitrary characteristic
(by Lipman~\cite[Theorem~2.1]
{LipmanVanishingForFinitelySupported2009}).
Our second result (which works in arbitrary characteristic in
dimension three) is the following:

\begin{theorem}
\label{theorem:lipman}
Adopt the notation and the hypothesis of Theorem~\ref{theorem:itoh}. Then 
$\homology^2_E(X, \strSh_X(m)) = 0$ for every $m \in \ints$. In particular,
if $R$ is regular and $I$ is such that $X$ is pseudo-rational. Then
$\widetilde{I^n} = I\widetilde{I^{n-1}}$ for every $n \geq \ell(I) = 3$.
\end{theorem}

The first assertion of the above theorem is not
exactly a special case of~\cite[Conjecture~(2.2)]{LipmanAdjoints94}, while
the second one is an affirmative case of~\cite[Conjecture~(1.6)]{LipmanAdjoints94}.

A word about the proofs. First we prove the vanishing
$\homology^3_\frakn(A) = 0$ asserted as part of
Theorem~\ref{theorem:itoh}\eqref{enum:thmitohHthreelocalvanishes};
this result is vital to our proof and connects Theorems~\ref{theorem:itoh}
and~\ref{theorem:lipman}.
Theorem~\ref{theorem:itoh}\eqref{enum:thmitohHthreelocalvanishes}--%
\eqref{enum:thmitohexactseq} follow from this vanishing and some
homological algebra.
The non-vanishing of $\homology^1(X, \strSh_X(1))$ implies the
non-vanishing of $\homology^2(E, \strSh_E(1))$. This, taken with the
results obtained in Section~\ref{sec:CohOnPTwo}, gives the proof of
Theorem~\ref{theorem:itoh}\eqref{enum:thmitohEquichar}.
The first assertion of Theorem~\ref{theorem:lipman} is
deduced from the vanishing of $\homology^3_\frakn(A)$;
the second assertion follows from this, and a result of Lipman.

This paper is organized as follows. Various preliminary facts are reviewed
in Section~\ref{sec:prelims}. In Section~\ref{sec:CohOnPTwo}, we prove some
results on the cohomology groups of certain coherent sheaves on the
projective plane over a field. Section~\ref{sec:proofs} is devoted to the
proofs of the results asserted above. In Section~\ref{sec:furtherrmks}, we
collect some examples that illustrate the use of our results.

\section*{Acknowledgements}
We thank Jugal Verma for helpful comments throughout the preparation of
this manuscript, A.~J.~Parameswaran for examples that
clarified some of the arguments and Joseph Lipman for comments on an
earlier draft.

\section{Preliminaries}
\label{sec:prelims}

By standard arguments~\cite[Section~8.4]{SwHuIntCl06}, we will assume that
$\Bbbk$ is infinite, and hence that $I$ has a minimal reduction $(x,y,z)$.
(We note the lengths of the various finite-length modules and the
vanishing of cohomology and the invariants $\bar e_i(-)$ mentioned in
Section~\ref{sec:results} are unaffected.) Further, the ring $A$, the
schemes $X$ and $E$, the sheaves $\strSh_X(m)$ and the invariants 
$\bar e_i(-)$ do not depend directly on $I$ but only on $\overline{I}$;
hence we may replace $I$ by its minimal reduction $(x,y,z)$.

\begin{notationbox}
Let $(R, \frakm, \Bbbk)$ is a Noetherian
three-dimensional Cohen-Macaulay analytically unramified local ring and $I
= (x,y,z)$ an $\frakm$-primary $R$-ideal. 
Let $\scrR = \oplus_{n \in \naturals} \overline{I^n}t^n$,
$A = \oplus_{n \in \ints} \overline{I^n}t^n$ and $G = A/(t^{-1})$.
Write 
$\frakn$ for the $A$-ideal $ItA + (t^{-1})A$ and $G_+$ for the $G$-ideal
generated by the homogeneous elements of $G$ of positive degree. 
Write $X = \Proj \scrR$ and $E = \Proj G$; it is an effective
Cartier divisor of $X$, defined by the invertible sheaf of ideals
$\strSh_X(1) := I\strSh_X$. 
For any coherent sheaf $\calF$ on $X$ and $m
\in \ints$, we will write $\calF(m) = \calF \otimes_{\strSh_X}
(I\strSh_X)^m$.
For elements $a_1, \ldots, a_r$ of $A$, and an $A$-module $M$, write 
$K_\bullet(a_1, \ldots, a_r; M)$ for the Koszul complex
$K_\bullet(a_1, \ldots, a_r) \otimes_A M$ and 
$\homology_i(a_1, \ldots, a_r; M)$ for $\homology_i(K_\bullet(a_1, \ldots,
a_r; M))$. If $a_1, \ldots, a_r$ are homogeneous and $M$ is graded, then 
$K_\bullet(a_1, \ldots, a_r; M)$ can be considered as a complex of graded
$A$-modules with maps of degree zero.
\end{notationbox}

For every $m \in \ints$, there is an exact sequence
\begin{equation}
\label{equation:definingOE}
0 \to \strSh_X(m+1) \to \strSh_X(m) \to \strSh_E(m)  \to 0.
\end{equation}

Write $X' := \Proj \left(\oplus_{n \in \naturals} I^nt^n\right)$. Then
$I\strSh_{X'}$ is an ample invertible
sheaf~\cite[II,~(4.6.6),~(8.1.7)]{EGA}. 
Since $R$ is analytically unramified, the natural map $X \to X'$ is finite;
$f$ is the composite $X \to X' \to \Spec R$. Since $\strSh_X(1)$ is 
the pull-back of $I\strSh_{X'}$,
it is an ample invertible sheaf on $X$.

Write 
$C^\bullet = 0 \to C^0 \to C^1 \to C^2 \to C^3 \to C^4 \to 0$ 
for the stable Koszul complex (sometimes also
called (extended) {\v{C}ech} complex) on $A$ with
respect to the sequence $t^{-1}, xt,yt,zt$. The modules $C^j$ are graded,
with maps of degree zero. 
Let $L_\bullet = 0 \to L_n \to \cdots \to L_0 \to 0$ be a
complex of graded $A$-modules with maps of degree zero. Then we get a
second-quadrant double-complex with $L_i \otimes_A C^j$ at position
$(-i,j)$. Note that all the maps in this complex are of degree zero. There
are two associated spectral sequences $'\!E_*^{*,*}$ and 
$''\!E_*^{*,*}$ with 
\begin{align*}
'\!E_1^{-i,j} = \homology^j_{\frakn}(L_i) \;\text{with maps}\;&
'\!d_1^{-i,j} : {'\!E_1^{-i,j}} \to {'\!E_1^{-i+1,j}}
\qquad \text{and}\\
''\!E_2^{-i,j} = \homology^j_{\frakn}(\homology_i(L_\bullet))
\;\text{with maps}\;&
''\!d_2^{-i,j} : {''\!E_2^{-i,j}} \to {''\!E_2^{-i-1,j+2}}.
\end{align*}
Since all the maps in the double-complex have degree zero, the maps in the
two spectral sequences above have degree zero. Now, suppose that 
$L_\bullet = K_\bullet(a_1, \ldots, a_r;M)$ for some homogeneous elements
$a_1, \ldots, a_r$ and a graded $A$-module $M$.
Then the $'\!E_1$ page
consists of the Koszul complexes 
$K_\bullet(a_1, \ldots, a_r; \homology^j_\frakn(M))$.
Hence $'\!E_2^{-i,j} = \homology_i(a_1, \ldots, a_r; \homology^j_\frakn(M))$.

\begin{discussionbox}
\label{discussionbox:itohfacts}
We quote some facts, chiefly from Itoh's papers. While some of these
statements hold \textit{mutatis mutandis} in more generality, we keep to
dimension three.
There is an exact sequence (see~\cite[Appendix~2]{ItohIntClRegSeq88})
\[
\cdots \to \homology^i_{\frakn}(A) \to \homology^i_{It}(A) \to
\homology^i_\frakm(R)[t,t^{-1}] \to \cdots,
\]
from which we see that the natural
map $\homology^i_{\frakn}(A) \to \homology^i_{It}(A)$ is an isomorphism
of graded $A$-modules, for $i=0,1,2$. There is an exact sequence of graded
$A$-modules:
\begin{equation}
\label{equation:itohTopExactSeq}
0 \to \homology^3_{\frakn}(A) \to \homology^3_{It}(A) \to
\homology^3_\frakm(R)[t,t^{-1}] \to 
\homology^4_{\frakn}(A) \to 0.
\end{equation}
For $i \geq 1$, there is an isomorphism $\homology^{i+1}_{It}(A) \to
\bigoplus_{n \in \ints} \homology^i(X, \strSh_X(n))$. We have that 
(see~\cite[Lemma~4(i), p.~391]{ItohIntClRegSeq88})
$t^{-1}, zt$ is a regular sequence on $A$, so
\begin{align}
\label{equation:lowestTwolocalCohVanish}
\homology^i_\frakn(A) &= 0 \;\text{for}\; i \leq 1.
\intertext{Additionally, (see~\cite[Theorem~2, p.~390]{ItohIntClRegSeq88}
and~\cite[Theorem~1.2]{HongUlriPSID06})}
\label{equation:itohHongUlrichHTwo}
\homology^2_\frakn(A)_n &= 0 \;\text{for every}\; n \leq 0.
\end{align}
Since $\homology^1_\frakm(R) = \homology^2_\frakm(R) =0$, we see that 
\begin{equation}
\label{equation:negativeHOneOX}
\homology^1(X,\strSh_X(n)) = 
\homology^2_{It}(A)_n = 0 \;\text{for every}\; n \leq 0.
\end{equation}

Further $\bar{e}_3(I) = \length_R(\homology^3_{It}(A)_0) =  
\length_R(\homology^2(X, \strSh_X))$~\cite[p.~114]{ItohNormalHilbCoeffs92}.
Finally we remark that when 
$\bar{e}_3(I) = \length_R(\homology^2(X, \strSh_X)) = 0$,
$A$ is Cohen-Macaulay if and only if 
$\overline{I^{n+2}} = I^n \overline{I^2}$ for all $n \geq 0$;
see~\cite[Corollary~16]{ItohNormalHilbCoeffs92}. (Itoh uses this condition
to state his conjecture and many results in his paper.) 
We note that 
$\overline{I^{n+2}} = I^n \overline{I^2}$ for all $n \geq 0$ if and only if
$\overline{I^{n+1}} = I \overline{I^n}$ for all $n \geq 2$. This can be
proved by induction on $n$.

We remark that~\eqref{equation:itohHongUlrichHTwo} holds
in dimension two also; we will need this 
in the proof of Lemma~\ref{lemma:generalFacts}%
\eqref{enum:generalHonequotient}.
\end{discussionbox}

\section{Coherent sheaves on the projective plane}
\label{sec:CohOnPTwo}

In this section let $\Bbbk$ be any field. Let $\calF$ be a 
coherent sheaf on $\projective^2 := \projective^2_\Bbbk$. We use
Boij-S\"oderberg theory for coherent sheaves on projective
spaces~\cite{EiScSupNat09}
to make some observations about
$\dim_\Bbbk \homology^1(\projective^2, \calF(m)), m > 0$,
for certain sheafs $\calF$ (see~\eqref{equation:cohTableOfFNEW}). We begin
with a review of Boij-S\"oderberg theory, keeping ourselves to
$\projective^2$. 

Let $\calF$ be a coherent sheaf on $\projective^2$. The \define{cohomology
table} of $\calF$ is the element 
$\gamma(\calF)$ of $\prod_{m = -\infty}^\infty
\reals^3$ with $(\gamma(\calF))_{i,m} = 
\dim_\Bbbk \homology^i(\calF(m))$. 
Say that $\calF$ has 
\define{super-natural cohomology} if 
\begin{enumerate}
\item for every $m \in \ints$, there exists at most one $i$ such that
$\homology^i(\calF(m)) \neq 0$; and,
\item the Hilbert polynomial of $\calF$ i.e., the polynomial that gives
the function 
\[
m \mapsto \sum_{i=0}^2 (-1)^i \dim_\Bbbk \homology^i(\calF(m))
\]
 has distinct integral zeros.
\end{enumerate}
If $\calF$ has super-natural cohomology, then we denote the zeros of its
Hilbert polynomial by $z_1 > \cdots > z_s$ (with $s=1$ or $s=2$), and
call the sequence $(z_1, \ldots, z_s)$ the
\define{zero sequence} of $\calF$. Given any sequence 
$z = (z_1 > \cdots > z_s)$ of integers, (with $s=1$ or $s=2$), 
set $z_{s+1} = \cdots = z_3 = -\infty$ and $z_0 = \infty$.
Then there exists a sheaf $\calF$ with super-natural cohomology satisfying
the following:
\[
\dim_\Bbbk \homology^i(\calF(m)) = 
\begin{cases}
\prod_{i=1}^s|(m-z_i)|, & z_{i+1} < m < z_i; \\
0, & \text{otherwise.}
\end{cases}
\]
(To get this, apply~\cite[Theorem~6.1]{EiScConjOfBS07} with
$m_1=\cdots=m_k=1$.) 
We write $\gamma^z = \gamma(\calF)$ for this $\calF$.
Moreover, such a sheaf $\calF$ may be taken to be locally free on a linear
subvariety $\projective^s_\Bbbk \subseteq \projective^2$. 

Write $S=\Bbbk[u,v,w]$ for the homogeneous coordinate ring of
$\projective^2_\Bbbk$ and ${S_+}=(u,v,w)S$. 
We are interested in a coherent sheaf $\calF$ on $\projective^2$ such that
\begin{equation}
\label{equation:cohTableOfFNEW}
\homology^j(\projective^2,  \calF(m)) = 0 \;\text{if}\;
\begin{cases}
j=0 \;\text{and}\; m \leq -1,\\
\text{or}, 
j=1 \;\text{and}\; m = 0,\\
\text{or}, 
j=2 \;\text{and}\; m \geq 0.\\
\end{cases}
\end{equation}

Take the map 
$\strSh_{\projective^2}(-1)^{\oplus 3} \to \strSh_{\projective^2}$
coming from the natural Koszul complex on $\projective^2$ given by $u, v,
w$, and apply $-\otimes \calF(m)$.
Write $\psi_m$ for the ensuing map 
$\homology^1(\calF(m-1)^{\oplus 3}) \to \homology^1(\calF(m))$.

\begin{proposition}
\label{proposition:BSthyCorNEW}
Let $\calF$ satisfy~\eqref{equation:cohTableOfFNEW}. 
Then:
\begin{enumerate}

\item \label{proposition:BSthyCorNEWSurj}
For every $m \geq 2$, $\psi_m$ is surjective.

\item \label{proposition:BSthyCorNEWPersistence}
For all $m \geq 1$, if $\homology^1(\calF(m)) = 0$ then 
$\homology^1(\calF(m+1)) = 0$.

\item \label{proposition:BSthyCorNEWNonNegRationals}
There exist a positive integer $r$ and 
non-negative rational numbers $a_1, \ldots, a_r$
such that
$\dim_\Bbbk \homology^1(\calF(1)) = \sum_{i=1}^ria_i$ 
and  
$\dim_\Bbbk \ker \psi_2 = \sum_{i=1}^r(i+2)a_i$.
Hence $\psi_2$ is injective if and only if 
$\homology^1(\calF(1)) = 0$.

\item \label{proposition:BSthyCorNEWVanishAtDegThree}
Suppose that 
$\dim_\Bbbk \homology^1(\calF(1)) = 
\dim_\Bbbk \homology^1(\calF(2)) = 1$. Then 
$\dim_\Bbbk \homology^1(\calF(m)) = 0$ for every $m \geq 3$. 

\item \label{proposition:BSthyCorNEWLengthAtLeastThree}
Suppose that $\psi_2$ is not injective. Then 
$\dim_\Bbbk \ker \psi_2 \geq 3$. 
\end{enumerate}
\end{proposition}

\begin{proof}
\ref{proposition:BSthyCorNEW}\eqref{proposition:BSthyCorNEWSurj}:
Let $m \geq 2$. Let
\[
0 \to \strSh_{\projective^2}(-3) \to 
\strSh_{\projective^2}(-2)^{\oplus 3} \to 
\strSh_{\projective^2}(-1)^{\oplus 3} \to 
\strSh_{\projective^2} \to 0
\]
be the Koszul complex on $\projective^2$, which is a locally free
resolution of $0$. Apply 
$- \otimes \calF(m)$ to this complex 
and consider the hypercohomology
spectral sequence for the functor $\homology^0(X, -)$.
The ${'\!E_1}$-page of this spectral sequence is
\[
\xymatrix@R=0.5em{
\homology^2(\calF(m-3)) \ar[r] & 0 \ar[r] & 0 \ar[r] & 0 \\
\homology^1(\calF(m-3)) \ar[r] & 
\homology^1(\calF(m-2))^{\oplus 3} \ar[r] & 
\homology^1(\calF(m-1))^{\oplus 3} \ar[r]^-{\psi_m} & 
\homology^1(\calF(m)) 
\\
\homology^0(\calF(m-3)) \ar[r] & 
\homology^0(\calF(m-2))^{\oplus 3} \ar[r] & 
\homology^0(\calF(m-1))^{\oplus 3} \ar[r] & \homology^0(\calF(m)) 
}
\]
(We have used the fact that $m \geq 2$ to get the three zeros in the top
row.) This spectral sequence abuts to zero, whence follows the surjectivity
of $\psi_m$.

\ref{proposition:BSthyCorNEW}\eqref{proposition:BSthyCorNEWPersistence}: Immediate
from~\ref{proposition:BSthyCorNEW}\eqref{proposition:BSthyCorNEWSurj}.

\ref{proposition:BSthyCorNEW}\eqref{proposition:BSthyCorNEWNonNegRationals}:
Let $\calZ$ be a chain of zero sequences and $a_z, z \in \calZ$
be non-negative real numbers such that 
\[
\gamma(\calF) = \sum_{z \in \calZ} a_z \gamma^z;
\]
see~\cite[Theorem~0.2]{EiScSupNat09} and the discussion leading to it.
Since we are over $\projective^2$, $z$ could be of the form $z = (z_1)$ or
of the form $z = (z_1 > z_2)$. In the former case, 
$\gamma^z$ is the cohomology table of a locally free 
sheaf on a projective line $\projective^1$ embedded inside $\projective^2$ 
as a linear subspace; in the latter case, it is the cohomology table
of a locally free sheaf on $\projective^2$. First suppose that $z = (z_1)$
and that $a_z > 0$. Then $\homology^1(\calF(m)) \neq 0$ for all $m < z_1$
and $\homology^0(\calF(m)) \neq 0$ for all $m > z_1$. It follows at once
from~\eqref{equation:cohTableOfFNEW} that $z_1 \in \{0,-1\}$. Such
cohomology tables do not contribute to $\homology^1(\calF(m))$ for any $m >
0$.
Now suppose that
$z = (z_1 > z_2)$ and that $a_z > 0$. Then for every $m$ such that 
$z_2 < m < z_1$, $\homology^1(\calF(m)) \neq 0$. It follows again 
from~\eqref{equation:cohTableOfFNEW} that $z_2 \leq 0$, that $z_1 \geq -1$ and
that if $\gamma^z$ contributes to $\homology^1(\calF(m))$ for any $m > 0$,
then $z_2 = 0$. Hence we need only consider the finite subset of $\calZ$
consisting of all the zero sequences of the form $(0 < z_1)$. Label these
zero sequences $z^{(i)} = (0 < i+1), i = 1, \ldots, r$. By abusing
notation, write $a_i$ for the non-negative real coefficient with which
$\gamma^{z^{(i)}}$ appears in $\gamma(\calF)$. These coefficients are in
fact rational~\cite[Theorem~0.2]{EiScSupNat09}. 
We may assume that $a_r > 0$.

Therefore
\begin{equation}
\label{equation:HOneFmNEW}
\dim_\Bbbk \homology^1(\calF(m)) = 
\begin{cases}
\sum_{i=m}^r m(i+1-m)a_i, & 1 \leq m \leq r,\\
0, & m > r.
\end{cases}
\end{equation}
In particular,
\begin{align*}
\dim_\Bbbk \homology^1(\calF(1)) & = \sum_{i=1}^ria_i \\
\dim_\Bbbk \homology^1(\calF(2)) & = \sum_{i=2}^r2(i-1)a_i.
\end{align*}
Hence 
\[
\dim_\Bbbk \ker \psi_2 = 
3 \sum_{i=1}^ria_i - \sum_{i=2}^r2(i-1)a_i
= \sum_{i=1}^r(i+2)a_i.
\]

\ref{proposition:BSthyCorNEW}\eqref{proposition:BSthyCorNEWVanishAtDegThree}:
First we show that $\dim_{\Bbbk}\homology^1(\calF(m)) \leq 1$ for every $m
\geq 3$.
Since $\psi_2$ is surjective, we see
that, without loss of generality, the multiplication map 
\[
\homology^1(\calF(1)) \stackrel{u}{\to} \homology^1(\calF(2))
\]
is non-zero; since $\dim_\Bbbk \homology^1(\calF(1))=
\dim_\Bbbk \homology^1(\calF(2))=1$, this map is an isomorphism. Write $H$
for the hyperplane (isomorphic to $\projective^1$) defined by the equation
$u=0$ inside $\projective^2$.
Use~\eqref{equation:cohTableOfFNEW} to see that 
$\homology^1(H, \calF|_H(2)) = 0$.  Hence 
$\homology^1(H, \calF|_H(m)) = 0$ for all $m \geq 3$, from which we see
that the map
\[
\homology^1(\calF(m-1)) \stackrel{u}{\to} \homology^1(\calF(m))
\]
is surjective for every $m \geq 3$. Therefore 
$\dim_\Bbbk \homology^1(\calF(m)) \leq 1$ for every $m \geq 3$. 

Now let $m_0 := \max\{m : \homology^1(\calF(m)) \neq 0\}$.
Then~\ref{proposition:BSthyCorNEW}\eqref{proposition:BSthyCorNEWPersistence}
implies that 
\[
\{m: m \geq 1, \homology^1(\calF(m)) \neq 0 \} = \{1, \ldots, m_0\}.
\]
We want to show that $m_0 \leq 2$. Let $r, z^{(i)}, a_i$ be as in the
proof
of~\ref{proposition:BSthyCorNEW}\eqref{proposition:BSthyCorNEWNonNegRationals}.
Then~\eqref{equation:HOneFmNEW}
implies that $m_0 = r$; hence we want to show that $r \leq 2$.
Note that
\begin{align*}
\dim_\Bbbk \homology^1(\calF(r)) & = ra_r =1\\
\dim_\Bbbk \homology^1(\calF(r-1)) & = (r-1)a_{r-1}+ 2(r-1)a_r = 1.
\end{align*}
Since $a_{r-1} \geq 0$ and $a_r > 0$, we see that 
$r \geq 2(r-1)$, or equivalently, that $r \leq 2$.

\ref{proposition:BSthyCorNEW}\eqref{proposition:BSthyCorNEWLengthAtLeastThree}: 
From~\ref{proposition:BSthyCorNEW}\eqref{proposition:BSthyCorNEWNonNegRationals},
we see that there exists
$i$ such that $a_i >0$, so 
$\dim_\Bbbk \ker \psi_2 > \dim_\Bbbk \homology^1(\calF(1)) > 0$;
since $\dim_\Bbbk \ker \psi_2$ and
$\dim_\Bbbk \homology^1(\calF(1))$ are integers,
$\dim_\Bbbk \ker \psi_2 \geq 2$. We will show that 
$\dim_\Bbbk \ker \psi_2 \neq 2$.

Suppose that $\dim_\Bbbk \ker \psi_2 = 2$.
Then 
by~\ref{proposition:BSthyCorNEW}\eqref{proposition:BSthyCorNEWNonNegRationals}, 
$\dim_\Bbbk \homology^1(\calF(1))=1$ and, hence,
$\dim_\Bbbk \homology^1(\calF(2))=1$. 
Thus
using~\ref{proposition:BSthyCorNEW}\eqref{proposition:BSthyCorNEWVanishAtDegThree}
we can refine the information from~\eqref{equation:cohTableOfFNEW} to the
following:
\begin{equation}
\label{equation:cohTableOfFNewNEW}
\homology^j(\projective^2,  \calF(m)) = 
\begin{cases}
0, & \;\text{if}\;
\begin{cases}
j=0 \;\text{and}\; m \leq -1,\\
\text{or}, 
j=1 \;\text{and}\; (m = 0 \;\text{or}\; m \geq 3),\\
\text{or}, 
j=2 \;\text{and}\; m \geq 0.\\
\end{cases}
\\
1, & \;\text{if}\; j = 1 \;\text{and}\; m \in \{1,2\}.\\
\end{cases}
\end{equation}
It suffices to show that there does not exist a coherent sheaf $\calF$
satisfying~\eqref{equation:cohTableOfFNewNEW}.

Assume to the contrary and 
let $M=\bigoplus_{m \in \ints}\homology^0(\projective^2, \calF(m))$. 
Recall that $S=\Bbbk[u,v,w]$ and ${S_+}=(u,v,w)S$. 
Consider the hypercohomology spectral sequence
from the proof of~\ref{proposition:BSthyCorNEW}\eqref{proposition:BSthyCorNEWSurj}, for $m=3$ and $m=4$.
For $m=3$, using~\eqref{equation:cohTableOfFNewNEW}, we get the
${'\!E_1}$-page to be
\[
\xymatrix@R=0.5em{
0 \ar[r] & 0 \ar[r] & 0 \ar[r] & 0 \\
0 \ar[r] & \Bbbk^{\oplus 3} \ar[r] & 
\Bbbk^{\oplus 3} \ar[r]^-{\psi_3} & 0 \\
\homology^0(\calF) \ar[r] & 
\homology^0(\calF(1))^{\oplus 3} \ar[r] & 
\homology^0(\calF(2))^{\oplus 3} \ar[r] & \homology^0(\calF(3)) .
}
\]
This abuts to zero, so the map $\Bbbk^{\oplus 3}  \to
\Bbbk^{\oplus 3}$ is surjective, hence also injective. Therefore the map 
$\homology^0(\calF(2))^{\oplus 3} \to \homology^0(\calF(3))$ is surjective,
i.e., 
\begin{equation}
\label{equation:NoMinGensInDegThreeNEW}
\dim_{\Bbbk}\left(\frac{M}{{S_+} M}\right)_3 = 0.
\end{equation}
For $m=4$, the ${'\!E_1}$-page is
\[
\xymatrix@R=0.5em{
0 \ar[r] & 0 \ar[r] & 0 \ar[r] & 0 \\
\Bbbk \ar[r] & \Bbbk^{\oplus 3} \ar[r] & 0 \ar[r]& 0\\
\homology^0(\calF(1)) \ar[r] & 
\homology^0(\calF(2))^{\oplus 3} \ar[r] & 
\homology^0(\calF(3))^{\oplus 3} \ar[r] & \homology^0(\calF(4)) 
}
\]
Since $\psi_2$ is surjective, at least one of the maps given by
multiplication by $u$, $v$, $w$ is non-zero, so the map 
$\Bbbk \to \Bbbk^{\oplus 3}$ of the second row in the above spectral
sequence is injective. Hence
\begin{equation}
\label{equation:TwoMinGensInDegFourNEW}
\dim_{\Bbbk}\left(\frac{M}{{S_+} M}\right)_4 = 2.
\end{equation}

Note that $\depth_{S_+}(M) = 2$~\cite[Theorem~13.21] {Loco24Hrs}.
Let 
\[
0 \to F_1 \stackrel{\partial}{\to} F_0 \to 0
\]
be a minimal graded $S$-free resolution of $M$. From this, 
$\Ext^1_S(M,S(-3))$ has a graded $S$-free presentation
\begin{equation}
\label{equation:MinPresOfHomNEW}
\Hom_S(F_0,S(-3)) \stackrel{\partial^*}{\to} \Hom_S(F_1,S(-3))
\; \text{with}\; \image(\partial^*) \subseteq 
S_+\cdot\left(\Hom_S(F_1,S(-3))\right).
\end{equation}
From this we get a graded isomorphism (of degree $0$)
\[
\Hom_S(F_1,S(-3)) \otimes_S \Bbbk = \Ext^1_S(M,S(-3))  \otimes_S \Bbbk.
\]

The Hilbert series of 
$\homology^2_{S_+}(M)$ is $\hssymb^1 + \hssymb^2 + \wp(\hssymb^{-1})$ where
$\wp(-)$ is a formal power series with a zero constant term
(see~\eqref{equation:cohTableOfFNewNEW}). Hence, by
(graded) local duality~\cite[Theorem~13.5]{Loco24Hrs} the Hilbert series of 
$\Ext^1_S(M,S(-3))$ is $\hssymb^{-2}+\hssymb^{-1} + \wp(\hssymb)$.
Let
\[
N := \Ext^1_S(M,S(-3))_{-2} \oplus \Ext^1_S(M,S(-3))_{-1}.
\]
Since $\Ext^1_S(M,S(-3))_{0} = 0$, $N$ is 
a summand of $\Ext^1_S(M,S(-3))$ as an $S$-module. 
Let
\[
G_\bullet : \quad 0 \to G_3 \to G_2 \to G_1 \to G_0 \to 0
\]
be a minimal graded $S$-free resolution of $N$.
We may extend~\eqref{equation:MinPresOfHomNEW}
to a (possibly non-minimal) graded $S$-free resolution of 
$\Ext^1_S(M,S(-3))$. Note that
$G_\bullet$ is a summand of this resolution of $\Ext^1_S(M,S(-3))$. In particular
$G_1$ is a summand of $\Hom_S(F_0,S(-3))$.

If $N = N_{-2} \oplus N_{-1}$ as $S$-modules, then 
\[
G_\bullet = K_\bullet(u,v,w; S(2) \oplus S(1))
\]
so $S(1)^{\oplus 3}$ is a summand of $G_1$ and hence of 
$\Hom_S(F_0,S(-3))$.  Equivalently,
$S(-4)^{\oplus 3}$ is a direct summand of $F_0$
contradicting~\eqref{equation:TwoMinGensInDegFourNEW}. On the other hand,
if $N$ is generated by $N_{-2}$ as an $S$-module, then without loss of
generality, 
\[
N = \left(\frac{S}{(u,v,w^2)}\right)(2), 
\]
so 
\[
G_\bullet = K_\bullet(u,v,w^2; S(2)),
\]
from which it follows that $S$ is a summand of $\Hom_S(F_0,S(-3))$.
Equivalently, $S(-3)$ is a summand of $F_0$,
contradicting~\eqref{equation:NoMinGensInDegThreeNEW}.
\end{proof}

\section{Proofs}
\label{sec:proofs}

As we remarked in Section~\ref{sec:results}, we first prove the vanishing
$\homology^3_\frakn(A) = 0$, whence, using some homological algebra,
follow
Theorem~\ref{theorem:itoh}\eqref{enum:thmitohHthreelocalvanishes}--%
\eqref{enum:thmitohexactseq}.
The non-vanishing of $\homology^2(X, \strSh_X(1))$, taken with the
results obtained in the previous section, gives the proof of
Theorem~\ref{theorem:itoh}\eqref{enum:thmitohEquichar}.
The first assertion of Theorem~\ref{theorem:lipman} is
deduced from the vanishing of $\homology^3_\frakn(A)$.
The rest of the proofs are more or less standard computations using
spectral sequences.

In the next lemma, we do not need the assumption that $\bar e_3(I) = 0$.
Thereafter that $\bar e_3(I) = 0$ will be the standing hypothesis.

\begin{lemma}
\label{lemma:generalFacts}
With notation as above:
\begin{asparaenum}
\item \label{enum:generalHonequotient}
$\homology^1_{\frakn}\left(\frac{A}{(t^{-1}, zt)}\right)_n = 0$ for all
$n \leq -1$.
\item \label{enum:generalhThreeVanishNegDegs}
$\homology_2(t^{-1},zt; \homology^3_\frakn(A))_n = 0$ 
for all $n \leq -1$.
\end{asparaenum}
\end{lemma}

\begin{proof}
\ref{lemma:generalFacts}\eqref{enum:generalHonequotient}:
By replacing $R$ by a rational extension and $x, y, z$ by suitable
elements, we may assume that, with $R' = R/(z)$ and $J = IR'$,
$\overline{J^n} = \overline{I^n}R' = \overline{I^n}/z\overline{I^{n-1}}$ 
for $n=1$ and for $n \gg 0$;
see~\cite[Theorem~1]{ItohNormalHilbCoeffs92}. 
Let $B = \oplus_{n \in \ints}\overline{J^n}t^n$. Consider the natural map
$\varphi : A/(zt) \to B$. By above, this map is bijective in degree $n$ for
$n \leq 1$ and for $n \gg 0$. Hence $\ker \varphi$ and $\coker \varphi$ are 
finite-length $A$-modules. (See also the first paragraph of the
proof~\cite[Theorem~2.1]{HongUlriPSID06}.) On the other hand, 
by~\eqref{equation:lowestTwolocalCohVanish}, $\depth_\frakn(A/(zt)) \geq 1$, so $\varphi$ is
injective.

Write $M = \coker \varphi$. 
Set $N_i = \Tor_i^A(M, A/(t^{-1})), i \geq 0$. 
Since $t^{-1}$ is a non-zero-divisor on $A$ and $B$, we have an exact
sequence
\[
0 \to N_1 \to A/(t^{-1}, zt) \to B/(t^{-1}) \to N_0 \to 0.
\]
Upon breaking this into two short exact sequences and applying 
$\homology^0_{\frakn}(-)$, we obtain the exact sequence
\[
0 \to N_0 \to
\homology^1_{\frakn}\left(\frac{A}{\left(t^{-1},zt\right)}\right) \to
\homology^1_{\frakn}\left(\frac{B}{\left(t^{-1}\right)}\right) \to 0.
\]
Since $M_n = 0$ for all $n \leq 1$, $(N_0)_n = 0$ for all $n \leq 1$.
Hence to prove the proposition, it suffices to show that 
\[
\homology^1_{\frakn}\left(\frac{B}{\left(t^{-1}\right)}\right)_n = 0 
\]
for all $n \leq -1$. Note, however that 
\[
\homology^1_{\frakn}\left(\frac{B}{\left(t^{-1}\right)}\right)_n 
= \ker \left(
\homology^2_{\frakn}(B)_{n+1} \stackrel{t^{-1}}{\to} 
\homology^2_{\frakn}(B)_{n} \right).
\]
From~\eqref{equation:itohHongUlrichHTwo} (using the remark at the end of
Discussion~\ref{discussionbox:itohfacts}),
we know that 
$\homology^2_{\frakn}(B)_{n+1} = 0$ for all $n \leq -1$. 
This finishes the proof of~\ref{lemma:generalFacts}\eqref{enum:generalHonequotient}.

\ref{lemma:generalFacts}\eqref{enum:generalhThreeVanishNegDegs}:
Take $L_\bullet = K_\bullet(t^{-1}, zt; A)$. Then 
${'\!E_2^{-i,j}} = \homology_i(t^{-1}, zt; \homology^j_\frakn(A))$.
By~\eqref{equation:lowestTwolocalCohVanish},
${'\!E_2^{-i,j}} = 0$ unless $(i,j) \in \{0,1,2\} \times \{2,3,4\}$. All
the maps on the ${'\!E_2}$-page except possibly
${'\!d_2^{-2,3}}$ and ${'\!d_2^{-2,4}}$ are zero. 
Again %
by~\eqref{equation:lowestTwolocalCohVanish},
the only possibly non-zero terms on the ${''\!E_2}$-page are 
${''\!E_2^{0,j} = \homology^j_{\frakn}(A/((t^{-1},zt)A))}$ for
$j \in \{0,1,2\}$.
Comparing the two spectral sequences, we see that the associated graded
object of 
$\homology^1_{\frakn}(A/((t^{-1},zt)A))$ with respect to the 
${'\!E}$ spectral sequence is 
${'\!E_{2}^{-1,2}} \oplus \ker {'\!d_2^{-2,3}}$. 
Using Lemma~\ref{lemma:generalFacts}\eqref{enum:generalHonequotient}, we
conclude that $(\ker {'\!d_2^{-2,3}})_n = 0$ for all $n \leq -1$. 
On the other hand, 
for all $n \leq 0$, ${'\!d_2^{-2,3}}$ maps $({'\!E_{2}^{-2,3}})_n$
to $({'\!E_{2}^{0,2}})_n = 
\homology_0(t^{-1}, zt; \homology^2_\frakn(A))_n = 0$. Hence, for all
$n \leq -1$, 
$({'\!E_{2}^{-2,3}})_n = (\ker {'\!d_2^{-2,3}})_n = 0$, proving 
\ref{lemma:generalFacts}\eqref{enum:generalhThreeVanishNegDegs}.
\end{proof}

\begin{lemma}
\label{lemma:ItohHypImpliesHtwoVanishPositiveTwists}
If $\bar{e}_3(I) = 0$, then $\homology^3_{It}(A)_n = 0$ for all $n \geq 0$.
\end{lemma}

\begin{proof}
Note that $\bar{e}_3(I) = \length_R(\homology^3_{It}(A)_0)$; hence it
suffices to show that if $m \in \ints$ is such that 
$\homology^3_{It}(A)_m=0$, then 
$\homology^3_{It}(A)_n=0$ for all $n \geq m$. 
Since $It(A/(zt))$ is generated by $xt$ and $yt$, we see that 
$\homology^3_{It}(A/(zt))=0$; hence 
the multiplication map by $zt$ on $\homology^3_{It}(A)$ is surjective. 
\end{proof}

As we mentioned in the beginning of this section, the vanishing asserted in
the next proposition is in some sense the first step in the proofs of our
results.

\begin{proposition}
\label{proposition:ItohHypImpliesHthreeZero}
If $\bar{e}_3(I) = 0$, then $\homology^3_\frakn(A) = 0$.
Hence~\eqref{equation:itohTopExactSeq} yields an exact sequence
\[
0 \to  \homology^3_{It}(A) \to \homology^3_\frakm(R)[t,t^{-1}] \to 
\homology^4_{\frakn}(A) \to 0.
\]
\end{proposition}

\begin{proof}
Note that since $\homology^3_\frakn(A)$ is a $(t^{-1},zt)$-torsion module,
it is zero if and only if its submodule
$(0 :_{\homology^3_\frakn(A)} (t^{-1}, zt)A)$ is zero.
Note that 
$(0 :_{\homology^3_\frakn(A)} (t^{-1}, zt)A)
=\homology_2(t^{-1},zt; \homology^3_\frakn(A))$. Therefore 
$(0 :_{\homology^3_\frakn(A)} (t^{-1}, zt)A)_n = 0$ for all $n \leq -1$ by 
Lemma~\ref{lemma:generalFacts}\eqref{enum:generalhThreeVanishNegDegs}, 
and for all $n \geq 0$ by~\eqref{equation:itohTopExactSeq} and
Lemma~\ref{lemma:ItohHypImpliesHtwoVanishPositiveTwists}.
\end{proof}

\begin{corollary}
\label{corollary:negativeHOneOE}
$\homology^2_{It}(G)_m = \homology^1(\strSh_E(m)) = 0$ for every $m \leq 0$.
\end{corollary}

\begin{proof}
Since $G = A/(t^{-1})$, there is an exact sequence
\[
\homology^2_{It}(A)(1) \stackrel{t^{-1}}{\to} \homology^2_{It}(A) \to 
\homology^2_{It}(G) \to 
\homology^3_{It}(A)(1) \stackrel{t^{-1}}{\to} \homology^3_{It}(A)
\]
with maps of degree zero. By
Proposition~\ref{proposition:ItohHypImpliesHthreeZero}, 
$t^{-1}$ is a non-zero-divisor on $\homology^3_{It}(A)$,
so we have a
surjective map $\homology^2_{It}(A)_m \to \homology^2_{It}(G)_m$  for every
$m \in \ints$. Now use~\eqref{equation:negativeHOneOX}.
\end{proof}

\begin{discussionbox}
\label{discussionbox:phim}
Consider the Koszul complex 
\[
K_\bullet: \;0 \to \strSh_X(-3) \to \strSh_X(-2)^{\oplus 3}  \to 
\strSh_X(-1)^{\oplus 3}  \to \strSh_X \to 0
\]
given by the section $(xt, yt, zt) \in \homology^0(X, \strSh_X(1))$.
Since $(x,y,z)$ is a reduction of
$\overline{I}$, this complex is exact. 
Therefore, for every $m \in \ints$,
$K_\bullet \otimes \strSh_X(m)$ is a locally free resolution of $0$, so for
every coherent $\strSh_X$-module $\calG$, 
the hypercohomology spectral sequences associated to the
complex $K_\bullet \otimes \calG(m)$
for the functor $\homology^0(X, -)$ abut to zero. 
Since $X$ has a covering by three affine open sets, $\homology^i(X, -) = 0$
for every $i \geq 3$.  
Hence the ${'\!E_1}$-page of this spectral sequence is
\[
\vcenter{\vbox{%
\xymatrix@R=0.5em{
\homology^2(\calG(m-3)) \ar[r] & 
\homology^2(\calG(m-2))^{\oplus 3} \ar[r] & 
\homology^2(\calG(m-1))^{\oplus 3} \ar[r] & 
\homology^2(\calG(m)) 
\\
\homology^1(\calG(m-3)) \ar[r] & 
\homology^1(\calG(m-2))^{\oplus 3} \ar[r] & 
\homology^1(\calG(m-1))^{\oplus 3} \ar[r]^-{\phi^{\calG}_m} & 
\homology^1(\calG(m)) 
\\
\homology^0(\calG(m-3)) \ar[r] & 
\homology^0(\calG(m-2))^{\oplus 3} \ar[r] & 
\homology^0(\calG(m-1))^{\oplus 3} \ar[r] & 
\homology^0(\calG(m)) 
}}}.
\]
We observe immediately that
\begin{equation}
\label{equation:HTwoVanishPersistence}
\text{if}\; \homology^2(\calG(m)) = 0 \;\text{then}\; \homology^2(\calG(m+1)) = 0.
\end{equation}
Further, if $\homology^2(\calG(m)) = 0$ for all $m \geq 0$, then 
the abutment of this spectral sequence to zero implies that
\begin{equation}
\label{equation:surjPhim}
\phi^{\calG}_m \;\text{is surjective for every}\; m \geq 2.
\end{equation}
Parenthetically, note that $\phi^{\strSh_X}_2$ is the map $\phi$ in the
exact sequence in
Theorem~\ref{theorem:itoh}\eqref{enum:thmitohexactseq}.
\end{discussionbox}

\begin{proposition}
\label{proposition:HtwoVanishGlobGen}
Let $\calG$ be a coherent sheaf on $X$ that is generated by global
sections, i.e., there exists a positive integer $r$ and a surjective
morphism of sheaves $\strSh_X^{\oplus r} \to \calG$. Then
$\homology^2(\calG(m)) = 0$ for all $m \geq 0$. Moreover, 
$\phi^{\calG}_m$ is surjective for every $m \geq 2$.
\end{proposition}

\begin{proof}
As $\homology^3(X,-)  = 0$, we have a surjective map 
\[
\homology^2(\strSh_X^{\oplus r}(m)) \to \homology^2(\calG(m))
\;\text{for every}\; m \in \ints.
\]
Since $\homology^2(\strSh_X(m)) = 0$ for every $m \geq 0$ (by hypothesis
and~\eqref{equation:HTwoVanishPersistence}), we get the first assertion.
The second one follows from~\eqref{equation:surjPhim}.
\end{proof}

\begin{proposition}
\label{proposition:HoneNonzeroOnContigSets}
We have the following:
\[
\{m : \homology^1(\strSh_X(m)) \neq 0\} = 
\{m : \homology^1(\strSh_E(m)) \neq 0\}.
\]
Moreover, if this set is non-empty, then it is of the form 
$\{1, \ldots, m\}$ for some positive integer $m$.
\end{proposition}

\begin{proof}
For each of these sets in question, if it is non-empty, then it is of the
form $\{1, \ldots, m\}$ for some positive integer $m$ (possibly different
in each case);
this follows immediately from~\eqref{equation:negativeHOneOX},
Corollary~\ref{corollary:negativeHOneOE} and
Proposition~\ref{proposition:HtwoVanishGlobGen}.

Recall that 
$\homology^1(\strSh_X(n)) = 0$ for all $n \gg 0$ and that
$\homology^2(\strSh_X(n)) = 0$ for all $n \geq 0$.
Further there is an exact sequence (see~\eqref{equation:definingOE})
\[
\to \homology^1(\strSh_X(n+1)) \to \homology^1(\strSh_X(n))
\to \homology^1(\strSh_E(n)) \to \homology^2(\strSh_X(n+1));
\]
Using these facts, one can immediately prove the following assertions:
if either of these sets given in the proposition is non-empty, then the
other is also non-empty; if the sets are non-empty, their maxima are same.
\end{proof}

We are now ready to prove Theorem~\ref{theorem:itoh}. Recall that we
have reduced the problem to the case of $I$ being generated by three
elements $x,y,z$.

\begin{proof}[Proof of Theorem~\ref{theorem:itoh}]
\eqref{enum:thmitohHthreelocalvanishes}: 
That $\homology^3_\frakn(A) = 0$ has been proved as part of
Proposition~\ref{proposition:ItohHypImpliesHthreeZero}. Since $E$ is a
Cartier divisor in $X$, $X$ is Cohen-Macaulay if and only if $E$ is
Cohen-Macaulay. Hence it suffices to show that $X$ is Cohen-Macaulay.
Let $m \in \naturals, m \geq 1$. Write $B^{(m)} = 
\oplus_{n \in \naturals} \overline{I^{mn}}t^n$ and
$A^{(m)} = \oplus_{n \in \ints} \overline{I^{mn}}t^n$.
Note that $X = \Proj B^{(m)}$. Write $u=(xt)^m, v=(yt)^m, w=(zt)^m$.
Since $x^m, y^m, z^m$ is a reduction for $\overline{I^m}$, we see that
\[
X = 
\Spec\left( (B^{(m)})_u \right) 
\bigcup 
\Spec\left( (B^{(m)})_v \right) 
\bigcup 
\Spec\left( (B^{(m)})_w \right).
\]
Notice that the natural map $(B^{(m)})_u \to (A^{(m)})_u$ is an
isomorphism; similarly for $v$ and $w$ also. Therefore, to show that $X$ is
Cohen-Macaulay, it suffices to show that $A^{(m)}$ is Cohen-Macaulay for
some $m$.

We claim that $A^{(m)}$ is Cohen-Macaulay for all $m \gg 0$. Write
$\frakn'$ for the homogeneous maximal ideal of $A^{(m)}$.  Notice that for
all $i$ and for all $m$,
\[
\homology^i_{\frakn'}(A^{(m)})_j = 
\begin{cases}
\homology^i_{\frakn}(A)_j, & \text{if}\; m \;\text{divides}\; j; \\
0, & \text{otherwise}.
\end{cases}
\]
Hence $\homology^i_{\frakn'}(A^{(m)}) = 0$ if $i=0,1,3$. To show that 
$\homology^2_{\frakn'}(A^{(m)}) = 0$ for all $m \gg 0$, note that there
exists $m_0$ such $\homology^2_{\frakn}(A)_j = 0$ for all $j \geq m_0$.
(For, $\homology^2_{\frakn}(A)_j = \homology^2_{It}(A)_j = 
\homology^1(X, \strSh_X(j))$ for all $j$ and $\strSh_X(1)$ is ample.) Hence
for all $m \geq m_0$, $\homology^2_{\frakn'}(A^{(m)}) = 0$.

\ref{theorem:itoh}\eqref{enum:thmitohboundForEtwoBar}:
It is known that 
\[
\sum_{i=0}^2 (-1)^i \bar e_i(I) \binom{n+2-i}{2-i}
= \sum_{i=0}^2 (-1)^i\length_R(\homology^i(E, \strSh_E(n)));
\]
see, for example,~\cite[p.~114]{ItohNormalHilbCoeffs92}. Therefore 
\[
\bar e_2(I) = \sum_{i=0}^2 (-1)^i\length_R(\homology^i(E, \strSh_E(-1))).
\]
We saw in Corollary~\ref{corollary:negativeHOneOE} that 
$\homology^1(E, \strSh_E(-1))=0$. Now consider the exact sequence
\[
0 \to \homology^0(X, \strSh_X) \to 
\homology^0(X, \strSh_X(-1)) \to 
\homology^0(E, \strSh_E(-1)) \to 
\homology^1(X, \strSh_X) \to  \cdots
\]
Note that 
$\homology^0(X, \strSh_X) = \homology^0(X, \strSh_X(-1)) = R$ and 
$\homology^1(X, \strSh_X) = 0$, to conclude that 
the finite-length module $\homology^0(E, \strSh_E(-1))$ is a quotient
of $R$ by a principal ideal, which implies that 
$\homology^0(E, \strSh_E(-1))=0$. Therefore
\[
\bar e_2(I) = \length_R(\homology^2(E, \strSh_E(-1))).
\]
The exact sequence used above continues further, and ends as follows: 
\[
\cdots \to \homology^2(X, \strSh_X) \to 
\homology^2(X, \strSh_X(-1)) \to 
\homology^2(E, \strSh_E(-1)) \to 0;
\]
therefore $\homology^2(X, \strSh_X(-1)) \simeq
\homology^2(E, \strSh_E(-1))$.
From the exact sequence in
Proposition~\ref{proposition:ItohHypImpliesHthreeZero}
we see that 
$\homology^2(X, \strSh_X(-1)) = \homology^3_{It}(A)_{-1} \subseteq
{\homology^3_\frakm(R)}$.
Now note that, since 
$\homology^2(X, \strSh_X(-1))  \simeq 
\homology^2(E, \strSh_E(-1))$, it is an $R/\overline{I}$-module.
Therefore $\homology^2(X, \strSh_X(-1))  \subseteq 
(0 :_{\homology^3_\frakm(R)} \overline{I})$.

\ref{theorem:itoh}\eqref{enum:thmitohexactseq}:
We see from Discussion~\ref{discussionbox:itohfacts} that the two 
given exact sequences are indeed restatements of each other;
hence we will work
with the description involving $\strSh_X$. 

The ${'\!E_1}$-page of the spectral sequence from
Discussion~\ref{discussionbox:phim} with $m=2$ and $\calG = \strSh_X$
becomes:
\[
\xymatrix@R=0.5em{
\homology^2(\strSh_X(-1)) \ar[r] & 0 \ar[r] & 0 \ar[r] & 0 
\\
0 \ar[r] & 0 \ar[r] & \homology^1(\strSh_X(1))^{\oplus 3}
\ar[r]^{\phi^{\strSh_X}_2} &
\homology^1(\strSh_X(2)) 
\\
0 \ar[r] & R \ar[r] & \overline{I}^{\oplus 3} \ar[r] & 
\overline{I^2}
}
\]
Note that $\phi^{\strSh_X}_2$ is the map $\phi$. The abutment of this
spectral sequence to zero gives the required exact sequence.

Suppose that $A$ is Cohen-Macaulay.  Then $\homology^1(X,\strSh_X(m)) =
\homology^2_{It}(A)_m = \homology^2_{\frakn}(A)_m = 0$ for every $m \in
\ints$. Conversely, if 
$\homology^1(X,\strSh_X(1))  = 0$, then the surjectivity of
$\phi_m^{\strSh_X}$ in Discussion~\ref{discussionbox:phim} 
for every $m \geq 2$ implies that 
$\homology^2_{\frakn}(A)_m =\homology^1(X,\strSh_X(m)) = 0$ for every $m
\geq 1$; that $\homology^2_{\frakn}(A)_m = 0$ for every $m \leq 0$
is~\eqref{equation:itohHongUlrichHTwo}.
Now
Theorem~\ref{theorem:itoh}\eqref{enum:thmitohHthreelocalvanishes}
implies that $A$ is Cohen-Macaulay. 

\ref{theorem:itoh}\eqref{enum:thmitohEquichar}:
Using the spectral sequence from Discussion~\ref{discussionbox:phim} with 
$(m,\calG) = (2,\strSh_X)$, $(m,\calG) = (3,\strSh_X)$ and 
$(m,\calG) = (2,\strSh_E)$
we have the following commutative diagram with exact columns and
surjective rows.
\begin{equation}
\label{equation:phidiag}
\vcenter{\vbox{%
\xymatrix@C=6em{
\homology^1(\strSh_X(2))^{\oplus 3} \ar[r]^{\phi^{\strSh_X}_3} \ar[d]_{\alpha_2}&
\homology^1(\strSh_X(3)) \ar[d]_{\alpha_3}\\
\homology^1(\strSh_X(1))^{\oplus 3} \ar[r]^{\phi = \phi^{\strSh_X}_2} \ar[d]_{\beta_1}&
\homology^1(\strSh_X(2)) \ar[d]_{\beta_2}\\
\homology^1(\strSh_{E}(1))^{\oplus 3} \ar[r]^{\phi^{\strSh_{E}}_2} \ar[d]&
\homology^1(\strSh_{E}(2)) \ar[d]\\
0 & 0
}}}
\end{equation}
Let $\eta \in \ker \beta_2 = \image{\alpha_3}$. Lift it to $\tilde{\eta}
\in \homology^1(\strSh_X(2))^{\oplus 3}$. Since
$\phi^{\strSh_X}_2(\alpha_2(\tilde{\eta})) = \eta$, we see that 
\[
\phi|_{\ker \beta_1} : {\ker \beta_1} \to {\ker \beta_2}
\]
is surjective. Snake lemma gives an exact sequence
\[
0 \to \ker\left(\beta_1|_{\ker \phi}\right) \to \ker \phi
\to  \ker \phi^{\strSh_{E}}_2 \to 0.
\]
It suffices to show that if $\homology^1(X,\strSh_X(1))  \neq 0$, then 
\begin{equation}
\label{equation:lengthKerForOE}
\length_R(\ker \phi_2^{\strSh_{E}}) \geq 3,
\end{equation}
under the hypothesis that $R$ is equicharacteristic or $\overline{I} =
\frakm$.

Note that in either of these cases, 
$G_0$ has a coefficient field, which too we denote by 
$\Bbbk$.
 Consider the $\Bbbk$-subalgebra $H$ of $G$ generated by 
$\overline{xt} := xt + \overline{I^2}, 
\overline{yt}, \overline{zt}$; since $x,y,z$ is a reduction for
$I$, $G$ is finite over $H$, and, therefore, $H$ is a polynomial ring.
The inclusion $H \subseteq G$ is a graded map under
which $\sqrt{H_+G} = \sqrt{G_+}$. Write  
$\projective^2 := \projective^2_\Bbbk = \Proj H$.
Hence we get a finite map $\nu : E \to \projective^2$.
Further, since $E$ is Cohen-Macaulay, the homogeneous localizations
$G_{(\overline{xt})}, G_{(\overline{yt})}, G_{(\overline{yt})}$ are 
finite Cohen-Macaulay modules over the corresponding homogeneous
localizations of $H$; therefore, by the Auslander-Buchsbaum formula,
$\nu_*\strSh_E$ is coherent locally free
$\strSh_{\projective^2}$-module which we will denote by $\calF$.
Further, 
$\strSh_E(1) = \nu^* \strSh_{\projective^2}(1)$. Using the projection
formula, we see that 
\[
\homology^j(E, \strSh_E(m)) = 
\homology^j(\projective^2,  \calF(m)) \;\text{for all $m$ and $j$}.
\]
Note that $\calF$ satisfies~\eqref{equation:cohTableOfFNEW}. 
Maps $\psi_2$ and $\phi^{\strSh_E}_2$ are the same.
Therefore~\eqref{equation:lengthKerForOE} follows from
Propositions~\ref{proposition:HoneNonzeroOnContigSets},
\ref{proposition:BSthyCorNEW}\eqref{proposition:BSthyCorNEWNonNegRationals}
and~\ref{proposition:BSthyCorNEW}\eqref{proposition:BSthyCorNEWLengthAtLeastThree}.
\end{proof}

\begin{remarkbox}
We could remove the restriction of equicharacteristic in
Theorem~\ref{theorem:itoh}\eqref{enum:thmitohEquichar} if we
had a weak version of Proposition~\ref{proposition:HoneNonzeroOnContigSets}
for $E_0 := \strSh_X/\frakm\strSh_X$. More precisely, suppose that we have
the following implication: if 
$\homology^1(\strSh_X(1)) \neq 0$ then  
$\homology^1(\strSh_{E_0}(1)) \neq 0$. Then, analogous
to~\eqref{equation:phidiag}, we have the diagram
\begin{equation}
\label{equation:phidiagCentral}
\vcenter{\vbox{%
\xymatrix@C=6em{
\homology^1(\frakm\strSh_X(1))^{\oplus 3} \ar[r]^{\phi^{\frakm\strSh_X}_2} 
\ar[d]_{\alpha_2}&
\homology^1(\frakm\strSh_X(2)) \ar[d]_{\alpha_3}\\
\homology^1(\strSh_X(1))^{\oplus 3} \ar[r]^{\phi = \phi^{\strSh_X}_2} \ar[d]_{\beta_1}&
\homology^1(\strSh_X(2)) \ar[d]_{\beta_2}\\
\homology^1(\strSh_{E_0}(1))^{\oplus 3} \ar[r]^{\phi^{\strSh_{E_0}}_2} \ar[d]&
\homology^1(\strSh_{E_0}(2)) \ar[d]\\
0 & 0
}}}
\end{equation}
with exact columns. 
Applying $-\otimes_R \strSh_X$ to any finite
$R$-free presentation of $\Bbbk$, we see that $\frakm \strSh_X$ is
generated by global sections; 
by Proposition~\ref{proposition:HtwoVanishGlobGen}, 
$\phi^{\frakm\strSh_X}_2$ is surjective, so the rows
of~\eqref{equation:phidiagCentral} are surjective. Now there is a finite
map $\nu : E_0 \to \projective^2$. Note that $\nu_*\strSh_{E_0}$ is a coherent
sheaf on $\projective^2$.  Arguing as in the proof of
Theorem~\ref{theorem:itoh}\eqref{enum:thmitohEquichar}, we
conclude that $\dim_\Bbbk \ker{\phi^{\strSh_{E_0}}_2} \geq 3$.
\end{remarkbox}

\begin{proof}[Proof of Corollary~\ref{corollary:cpr}]
Use
Theorem~\ref{theorem:itoh}\eqref{enum:thmitohboundForEtwoBar},
\ref{theorem:itoh}\eqref{enum:thmitohexactseq}
and~\ref{theorem:itoh}\eqref{enum:thmitohEquichar}.
\end{proof}

\begin{proof}[Proof of Theorem~\ref{theorem:lipman}]
Write $U = X \minus E = \Spec R \minus \{\frakm\}$. Then, for every $m \in
\ints$, we have exact sequences
\begin{equation}
\label{equation:lesHomologyWithSupports}
\cdots \to \homology^i_E(X, \strSh_X(m)) \to 
\homology^i(X, \strSh_X(m)) \to 
\homology^i(U, \strSh_X(m)|_U) \to \cdots.
\end{equation}
For every $m$, 
$\homology^1(U, \strSh_X(m)|_U) = 
\homology^1(U, \strSh_U) = \homology^2_{\frakm}(R) = 0$.
However we see from the exact sequence in
Proposition~\ref{proposition:ItohHypImpliesHthreeZero} that the natural map 
$\homology^2(X, \strSh_X(m)) \to \homology^2(U, \strSh_U) =
\homology^3_{\frakm}(R)$ is injective for
every $m$. Hence 
\begin{equation}
\label{equation:HtwoWithSupport}
\homology^2_E(X, \strSh_X(m)) = 0 \;\text{for every}\; m \in \ints. 
\end{equation}

We want to show that 
\begin{equation}
\label{equation:HoneWithSupport}
\homology^1_E(X, \strSh_X(m)) = 0 \;\text{for every}\; m \leq 0.
\end{equation}
From~\eqref{equation:lesHomologyWithSupports}
and~\eqref{equation:negativeHOneOX} we get an exact sequence
\[
0 \to \homology^0_E(X, \strSh_X(m)) \to 
\homology^0(X, \strSh_X(m)) \to 
\homology^0(U, \strSh_X(m)|_U) \to 
\homology^1_E(X, \strSh_X(m)) \to 0,
\]
for every $m \leq 0$.
For every $m \in \ints$, $\homology^0_E(X, \strSh_X(m)) = 0$, while for $m
\leq 0$, the middle two modules are isomorphic to $R$. Since 
$\homology^1_E(X, \strSh_X(m))$ has finite length for every $m \in \ints$,
we conclude~\eqref{equation:HoneWithSupport}. 

Now suppose that $R$ is regular. Let
$\omega_X := \homology^0 f^! R$; since $R$ is Gorenstein and $X$
Cohen-Macaulay, $\omega_X$ is dualizing sheaf on $X$.
Duality with supports~\cite[Theorem~p.188]{LipmanDesingTwoDim1978} applied
to \eqref{equation:HtwoWithSupport} and
\eqref{equation:HoneWithSupport} imply that
$\homology^1(X, \omega_X(m)) = 0$ for every $m \in \ints$ and that
$\homology^2(X, \omega_X(m)) = 0$ for every $m \geq 0$.
Therefore, the ${'\!E_1}$-page of the spectral sequence 
from Discussion~\ref{discussionbox:phim} with $m\geq3$ and $\calG =\omega_X$
has only one non-zero row:
\[
0 \to \homology^0( \omega_X(m-3)) \to 
\homology^0( \omega_X(m-2))^{\oplus 3} \to
\homology^0( \omega_X(m-1))^{\oplus 3} \to
\homology^0( \omega_X(m)) \to 0.
\]
Hence this is an exact sequence, from which we conclude that
\[
\homology^0(X, \omega_X(m)) = I \homology^0(X, \omega_X(m-1)) 
\;\text{for all}\; m \geq 3.
\]
Since $X$ is pseudo-rational, $\widetilde{I^m} = 
\homology^0(X, \omega_X(m))$ for all 
$m \geq 1$~\cite[(1.3.1)]{LipmanAdjoints94}.
\end{proof}

\section{Further remarks}
\label{sec:furtherrmks}

We now describe a few examples in which explicit calculations show that
$A$ (and $\scrR$, because of the vanishing of $\homology^2(X, \strSh_X)$)
is Cohen-Macaulay.

\begin{examplebox}
Let $(R, \frakm)$ be a three-dimensional regular local ring, and $I :=
(x,y,z)$ an $\frakm$-primary ideal. Suppose that $S := R/(x)$ is pseudo-rational
or that $R = \complex[[u,v,w]]$ and $S := R/(x)$ has an elliptic (Gorenstein)
singularity~\cite[Definition~4.12]{ReidChaptersOnSurf1997}. Let $A$ and $X$
be as earlier. We will show that $\homology^2(X, \strSh_X) = 0$ and that 
$\homology^1(X, \strSh_X(1)) = 0$; hence by
Theorem~\ref{theorem:itoh}\eqref{enum:thmitohexactseq}
$A$ is Cohen-Macaulay.

Write $J = IS$. Let 
$Y = \Proj \left(\oplus_{n \in \naturals} \overline{J^n}t^n\right)$. We
first show that
\begin{equation}
\label{equation:HOneYVanishesPosTwist}
\homology^1(Y, \strSh_Y(m))  = 0 \;\text{for every}\; m \geq 1.
\end{equation}
To prove this, consider the two cases.
If $S$ is pseudo-rational, then $\homology^1(Y, \strSh_Y)  = 0$. 
Since $(y,z)$ is a minimal reduction of $J$, the Koszul complex
\[
0 \to \strSh_Y(-2) \to \strSh_Y(-1)^{\oplus 2} \to \strSh_Y \to 0
\]
given by $(yt, zt) \in \homology^0(Y, \strSh_Y(1)^{\oplus 2})$ is exact.
Apply $-\otimes_{\strSh_Y} \strSh_Y(m)$ for a positive integer $m$ to 
get~\eqref{equation:HOneYVanishesPosTwist} in this case. 
If $S$ has elliptic singularity, 
$\dim_{\complex}\homology^1(Y, \strSh_Y)  = 1$. 
Let
$m$ be the largest integer such that $\homology^1(\strSh_Y(m)) \neq 0$.
Suppose that $m > 0$. Since $(y^m, z^m)$ is a minimal reduction of $J^m$, 
the Koszul complex
\[
0 \to \strSh_Y(-2m) \to \strSh_Y(-m)^{\oplus 2} \to \strSh_Y \to 0
\]
given by $(y^mt^m, z^mt^m) \in \homology^0(Y, \strSh_Y(m)^{\oplus 2})$ is
exact.
Apply $-\otimes_{\strSh_Y} \strSh_Y(2m)$ to get a surjective map
\[
\homology^1(\strSh_Y) \to \homology^1(\strSh_Y(m))^{\oplus 2}.
\]
Since $\dim_{\complex}\homology^1(Y, \strSh_Y)  = 1$, we see that 
$\homology^1(\strSh_Y(m))=0$ contradicting the hypothesis that $m > 0$, from
which \eqref{equation:HOneYVanishesPosTwist} follows.

Note that $Y$ is defined by $xt \in \homology^0(X, \strSh_X(1))$, so we
have an exact sequence
\[
0 \to \strSh_X(-1) \to \strSh_X \to \strSh_Y \to 0,
\]
whence we obtain an injective map $\homology^1(\strSh_X(1)) \to
\homology^1(\strSh_Y(1))$, showing that 
$\homology^1(\strSh_X(1))  = 0$. Similarly, for every $m \geq 0$, we get
isomorphisms
$\homology^2(\strSh_X(m)) \to \homology^2(\strSh_X(m+1))$, from which we
conclude that  $\homology^2(\strSh_X(m))  = 0$ for every $m \geq 0$.
\end{examplebox}

\begin{examplebox}
Let $(R, \frakm)$ be a three-dimensional Cohen-Macaulay analytically
unramified domain. Let $I$ be an $\frakm$-primary ideal. 
\emph{Suppose} that there exists a non-singular scheme $Z$ that is
projective over $R$ such that $I\strSh_Z$ is invertible.
As earlier, set $X = \Proj \left(\oplus_{n \in \naturals}
\overline{I^n}t^n\right)$, and let $f: X \to \Spec R$ be the natural
morphism. Then there is a morphism $g: Z \to X$ such that the morphism $Z
\to \Spec R$ is the composite $fg$.
We compute $\homology^*(Z, I^m\strSh_Z)$, $m \geq 0$, using the spectral
sequence for the composite $fg$.
Since $I\strSh_Z = g^*I\strSh_X$, we use the projection formula to see that 
$R^ig_*I^m\strSh_Z = (R^ig_*\strSh_Z)(m)$ for all integers $m$.
Thus the spectral sequence gives an exact sequence
\[
0 \rightarrow \homology^1(X,\strSh_X(m)) \rightarrow 
\homology^1(Z, I^m\strSh_Z) \rightarrow 
\homology^0(X, (R^1g_*\strSh_Z)(m)) \rightarrow 
\homology^2(X, \strSh_X(m)) \rightarrow 
\homology^2(Z, I^m\strSh_Z).
\]
Further suppose that $\homology^1(Z, \strSh_Z) = 
\homology^2(Z, \strSh_Z) = 0$. (This would hold, e.g., if $R$ has rational
singularities.) Thus we have an isomorphism
\[
\homology^0(X, R^1g_*\strSh_Z) \stackrel{\simeq}{\to}
\homology^2(X, \strSh_X).
\]
Now suppose that $\homology^1(Z, I^m\strSh_Z)  = 0$ for all $m \gg 0$.
Since $\homology^2(X, \strSh_X(m))  = 0$ for all $m \gg 0$, we see that 
$\homology^0(X, (R^1g_*\strSh_Z)(m)) = 0$ for all $m \gg 0$. Since 
$\strSh_X(1)$ is ample, this means that
$R^1g_*\strSh_Z = 0$, and, hence, that
$\homology^2(X, \strSh_X)=0$. 

In this context, suppose that $R$ is regular and that $I$ is finitely
supported, i.e., $Z$ can be obtained from $\Spec R$ by a finite sequence of
blow-ups at
closed points.  Then one has
that $\homology^1(Z, I^m\strSh_Z) = 0$ for every $m \geq
1$~\cite[Corollary~3.1(ii)]{LipmanVanishingForFinitelySupported2009}.
Hence $\homology^2(X, \strSh_X)= \homology^1(X, \strSh_X(m)) = 0$ for every
$m \geq 1$, and $A$ is Cohen-Macaulay.
\end{examplebox}

\begin{examplebox}
Let $(R, \frakm)$ a three-dimensional Cohen-Macaulay analytically
unramified ring such that $\bar e_3(\frakm) = 0$. Set $A =
\oplus_{n \in \ints}\overline{\frakm^n}t^n$ and 
$X = \Proj (\oplus_{n \in \naturals}\overline{\frakm^n}t^n)$. 
Suppose that the associated graded ring 
$\oplus_{n \in \naturals} \frakm^n/\frakm^{n+1}$ is reduced. Then
$\overline{\frakm^n} = \frakm^n$ for all $n \in \naturals$. (To see this,
suppose, by way of contradiction, that there exists $n \in \naturals$ and 
$a \in \overline{\frakm^n} \minus \frakm^n$. Let $d$ be the largest integer
such that $a \in \frakm^d$. Suppose that
\[
a^r + a_1a^{r-1} + \cdots + a_r = 0
\]
with $a_i \in \frakm^{ni}$ for all $1 \leq i \leq r$. Then $a^r \in
\frakm^{rd+(n-d)i} \subseteq \frakm^{rd+1}$, implying that the $a$ is
nilpotent in the associated graded ring. This argument works for all
(radical) ideals, not just $\frakm$.) Therefore $A = R[\frakm t, t^{-1}]$,
$X$ is the blow-up of $\Spec R$ at the closed point, and $E$ is reduced.
Suppose further that $R$ is normal. Then $E$ is connected. Thus
$\homology^0(E, \strSh_E)$ is a reduced local finite $\Bbbk$-algebra, i.e,
a finite field extension of $\Bbbk$. If additionally $\Bbbk$ is
algebraically closed, then 
$\homology^0(E, \strSh_E)=\Bbbk$, so
$\homology^1(X, \frakm\strSh_X) = 0$. 
By Theorem~\ref{theorem:itoh}\eqref{enum:thmitohexactseq}
$A$ is Cohen-Macaulay.
\end{examplebox}

\newcommand{\etalchar}[1]{$^{#1}$}
\def\cfudot#1{\ifmmode\setbox7\hbox{$\accent"5E#1$}\else
  \setbox7\hbox{\accent"5E#1}\penalty 10000\relax\fi\raise 1\ht7
  \hbox{\raise.1ex\hbox to 1\wd7{\hss.\hss}}\penalty 10000 \hskip-1\wd7\penalty
  10000\box7}

\end{document}